\documentclass[a4paper,oneside,reqno,11pt]{amsart}

\usepackage[utf8]{inputenc}
\usepackage[english]{babel}
\usepackage{amsfonts}
\usepackage{amsmath}
\usepackage{amssymb}
\usepackage{amsthm}
\usepackage{enumerate}
\usepackage{geometry}

\newcommand*{\ges}{\geqslant}
\newcommand*{\les}{\leqslant}
\newcommand*{\1}{1\!\!\,\mathrm{I}}
\newcommand*{\sign}{\mathrm{sign}\,}
\newcommand*{\abs}[1]{\left|#1\right|}
\newcommand*{\Prob}[1]{\mathbf{P} \left\lbrace #1\right\rbrace}
\newcommand*{\E}{\mathbf{E}}
\newcommand*{\Var}{\mathrm{Var}\,}
\newcommand*{\jqv}[2]{\left\langle #1,#2\right\rangle}

\newcommand*{\ve}{\varepsilon}
\newcommand*{\mbZ}{\mathbb{Z}}
\newcommand*{\mbR}{\mathbb{R}}
\newcommand*{\mcA}{\mathcal{A}}

\newcommand*{\mcN}{\mathcal{N}}
\newcommand*{\mbP}{\mathbf{P}}

\theoremstyle{plain}
\newtheorem{theorem}{Theorem}

\newtheorem{proposition}[theorem]{Proposition}
\newtheorem{corollary}[theorem]{Corollary}

\theoremstyle{definition}

\theoremstyle{remark}
\newtheorem{remark}[theorem]{Remark}

\title
[The CLT for the number of clusters of the Arratia flow]
{The central limit theorem for the number of clusters of the Arratia flow}
\author{E.~V.~Glinyanaya}
\address{E.~V.~Glinyanaya: Institute of Mathematics, National Academy of Sciences of Ukraine, Teresh\-chenkivska str.~3, Kiev~01004, Ukraine}
\email{glinkate@gmail.com}
\author{V.~V.~Fomichov}
\address{V.~V.~Fomichov: Institute of Mathematics, National Academy of Sciences of Ukraine, Teresh\-chenkivska str.~3, Kiev~01004, Ukraine}
\email{v-vfom@imath.kiev.ua}

\begin{document}

\begin{abstract}
In this paper we prove the central limit theorem for the number of clusters formed by the particles of the Arratia flow starting from the interval $[0;n]$ as $n\to\infty$ and obtain an estimate of the Berry--Esseen type for the rate of this convergence.
\end{abstract}

\maketitle

\section{Introduction}
\label{section1}

In this paper we consider the Arratia flow $\{x(u,\cdot),\; u\in\mbR\}$, which is an ordered family of standard Brownian motions starting from every point of the real line such that for any $u,v\in\mbR$ the joint quadratic variation of $x(u,\cdot)$ and $x(v,\cdot)$ is given by
\[
\jqv{x(u,\cdot)}{x(v,\cdot)}_t=\int\limits_0^t \1_{\{0\}}(x(u,s)-x(v,s))\,ds,\quad t\ges 0,
\]
where $\1_{\{0\}}$ stands for the indicator function of the set $\{0\}$. This flow was constructed by R.~A.~Arratia~\cite{Arratia} as a weak limit of families of coalescing simple random walks and can be informally described as a system of Brownian particles any two of which move independently until they meet, after which they coalesce and move together.

In~\cite{Harris} T.~E.~Harris considered a generalisation of the Arratia flow, in which the indicator function $\1_{\{0\}}$ is replaced by a non-negative definite function $\Gamma$, which is called the covariance function of the flow, and proved its existence under certain conditions on $\Gamma$.

In the same paper T.~E.~Harris proved that for the Arratia flow $\{x(u,\cdot),\; u\in\mbR\}$ for any time $t>0$ and interval $[u_1;u_2]\subset\mbR$ the set $x([u_1;u_2],t)$ is almost surely finite. From this it follows that for any time $t>0$ and interval $[u_1;u_2]$ the number
\[
\nu_t([u_1;u_2]):=\#\; x([u_1;u_2],t)
\]
of elements of the set $x([u_1;u_2],t)$ is almost surely finite (for a different proof see monograph~\cite{Dorogovtsev2007} of A.~A.~Dorogovtsev). R.~Tribe and O.~Zaboronski~\cite{TribeZaboronski} proved that for any $t>0$ the random point process $x(\mbR,t)$ is Pfaffian and found its kernel; basing on some of their formulae, the distribution of $\nu_t([0;u])$ was found in~\cite{Fomichov}. Earlier for Harris flows the necessary and sufficient condition of the existence of coalescence of particles and an estimate for the mean value of the number of clusters were obtained by H.~Matsumoto in~\cite{Matsumoto} respectively. For the Arratia flow the large deviation principle and the law of the iterated logarithm for the size of the cluster containing the point zero were established by A.~A.~Dorogovtsev and O.~V.~Ostapenko~\cite{DorogovtsevOstapenko} and A.~A.~Dorogovtsev, A.~V.~Gnedin and M.~B.~Vovchanskii~\cite{DorogovtsevGnedinVovchanskii} respectively.

Since the covariance of any two particles in Harris flows depends only on the distance between them, such flows are stationary with respect to the spatial variable. In~\cite{Glinyanaya}, under the assumption that the covariance function converges to zero at infinity, their ergodicity with respect to the spatial variable was established and an estimate for the strong mixing coefficient was found.

In this paper we prove the following central limit theorem for $\nu_t([0;n])$ as $n\to\infty$.

\begin{theorem}
\label{theorem1}
For any $t>0$
\[
\dfrac{\nu_t([0;n])-\E\nu_t([0;n])}{\sqrt{n}} \Longrightarrow \mcN(0;\sigma_t^2),\quad n\to\infty,
\]
where $\sigma_t^2:=\dfrac{3-2\sqrt{2}}{\sqrt{\pi t}}$.
\end{theorem}

Furthermore, we also obtain an estimate for the rate of this convergence by proving the following inequality of the Berry--Esseen type.

\begin{theorem}
\label{theorem2}
For any $n\ges 1$
\[
\sup_{z\in\mbR} \abs{\Prob{\dfrac{\nu_t([0;n])-\E\nu_t([0;n])}{\sqrt{n}}\les z}-\int\limits_{-\infty}^z \dfrac{1}{\sqrt{2\pi\sigma_t^2}} e^{-r^2/2\sigma^2_t}\,dr}\les Cn^{-1/2}(\log n)^2.
\]
\end{theorem}

Let us note that due to the scaling invariance of the Arratia flow (e.~g., see~\cite[subsection~2.3]{TribeZaboronski})
\begin{equation}
\label{equation0}
x(\cdot,\cdot)\stackrel{d}{=}\dfrac{1}{\ve} x(\ve \cdot,\ve^2 \cdot),\quad \ve>0,
\end{equation}
from Theorem~\ref{theorem1} the following corollary can be deduced.

\begin{corollary}
The following convergence in distribution takes place:
\[
\sqrt[4]{t} \cdot \nu_t([0,1])-\dfrac{1}{\sqrt[4]{t} \cdot \sqrt{\pi}} \Longrightarrow \mcN(1;\sigma^2),\quad t\to 0+,
\]
where $\sigma^2=\dfrac{3-2\sqrt{2}}{\sqrt{\pi}}$.
\end{corollary}

The main part of this paper consists of two sections. In Section~\ref{section2} we establish the asymptotic behaviour of the variance and all moments of $\nu_t([0;u])$ and in Section~\ref{section3} we give the proof of Theorems~\ref{theorem1} and~\ref{theorem2}.

\section{Asymptotics of the variance and moments of $\nu_t([0;u])$}
\label{section2}

In this section we establish the asymptotic behaviour of the variance and all moments of $\nu_t([0;u])$. Our proof is based on the results of R.~Tribe and O.~Zaboronski~\cite{TribeZaboronski}, and we refer the reader to this paper for the definitions of the objects we use in this section.

In paper~\cite{TribeZaboronski} the authors proved that for any time $t>0$ the clusters of the Arratia flow form a Pfaffian point process with the kernel
\[
K_t(u,v)=\dfrac{1}{\sqrt{t}} K\left(\dfrac{u}{\sqrt{t}}, \dfrac{v}{\sqrt{t}}\right),
\]
where
\[
K(u,v)=
\begin{pmatrix}
-F''(v-u) & -F'(v-u)\\
F'(v-u) & \sign(v-u) \cdot F(\abs{v-u})
\end{pmatrix}
\]
with the function $F$ given by
\[
F(z):=\dfrac{1}{\sqrt{\pi}} \int\limits_z^{+\infty} e^{-r^2/4}\,dr,\quad z>0.
\]
In particular, it means that for any $n\ges 1$ the $n$th factorial moment (for this moment we will use the notation $a^{[n]}=a(a-1) \ldots (a-n+1)$, $a\in\mbZ_+$) of the number $N_t([0;u])$ of particles of the Arratia flow which at time $t>0$ are found at the interval $[0;u]$ is given by
\[
\E N_t^{[n]}([0;u])=\int\limits_0^u \stackrel{n}{\ldots} \int\limits_0^u \rho_t^{(n)}(v_1,\ldots,v_n)\,dv_1 \ldots dv_n,
\]
where $\rho_t^{(n)}$ is the $n$-point density permitting the following representation:
\[
\rho_t^{(n)}(v_1,\ldots,v_k)=\mathrm{Pf}\left[K_t(v_i,v_j),\; i,j=1,\ldots,n\right],\quad v_1,\ldots,v_n\in\mbR.
\]

To obtain the expressions for the moments of $\nu_t([0;u])$ it remains to note that
\begin{equation}
\label{equation}
\nu_t([0;u])\stackrel{d}{=} N_t([0;u])+1,
\end{equation}
which can be easily proved with the help of the dual flow (e.~g., see~\cite{TothWerner}, \cite{Dorogovtsev}, \cite[subsection~2.2]{TribeZaboronski}). Recall that for fixed time $t_0>0$ the dual flow is a system $\{y(u,t),\; u\in\mbR,\; 0\les t\les t_0\}$ of coalescing Brownian motions in backward time starting from every point of the real line characterised by the property that its trajectories do not intersect those of the particles of the restriction $\{x(u,t),\; u\in\mbR,\; 0\les t\les t_0\}$ of the Arratia flow to the time interval $[0;t_0]$. It is known that $\{y(u,t),\; u\in\mbR,\; 0\les t\les t_0\}$ agrees in distribution with $\{x(u,t),\; u\in\mbR,\; 0\les t\les t_0\}$, and equality~\eqref{equation} follows from the fact that the set $y(\mbR,t)$ coincides with the set of points of discontinuity of the mapping $x(\cdot,t) \colon \mbR \rightarrow \mbR$.

\begin{proposition}
For any $t>0$ and $u>0$ we have
\[
\Var\nu_t([0;u])=-\dfrac{4}{\pi}+\dfrac{3u}{\sqrt{\pi t}}+ \dfrac{4}{\pi}e^{-u^2/2t}-\dfrac{2}{\pi}\int\limits_0^{u/\sqrt{t}} e^{-z^2/4}\,dz-\dfrac{4u}{\pi\sqrt{t}}\int\limits_0^{u/\sqrt{t}}e^{-z^2/2}\,dz.
\]
\end{proposition}

\begin{proof}
First of all, let us note that
\[
\E N_t([0;u])=\E N_t^{[1]}([0;u])=\int\limits_0^u \rho_t^{(1)}(v)\,dv= \dfrac{u}{\sqrt{\pi t}},
\]
and so
\begin{equation}
\label{equation1}
\E\nu_t([0;u])=1+\dfrac{u}{\sqrt{\pi t}}.
\end{equation}
Moreover, on the one hand,
\begin{equation}
\label{equation2}
\E N_t^{[2]}([0;u])=\E\nu_t^2([0;u])-3\E\nu_t([0,u])+2,
\end{equation}
and, on the other hand,
\begin{equation}
\label{equation3}
\E N_t^{[2]}([0;u])=\int\limits_0^u \int\limits_0^u \rho^{(2)}_t(v_1,v_2)\,dv_1dv_2,
\end{equation}
where (for notational simplicity here and below for antisymmetric matrices we omit their entries below the diagonal)
\begin{gather*}
\rho^{(2)}_t(v_1,v_2)=
\mathrm{Pf}
\left[
\begin{matrix}
0 & \dfrac{1}{\sqrt{\pi t}} & -\dfrac{v_2-v_1}{2\sqrt{\pi} \cdot t} e^{-(v_2-v_1)^2/4t} & \dfrac{1}{\sqrt{\pi t}} e^{-(v_2-v_1)^2/4t}\\
& 0 & -\dfrac{1}{\sqrt{\pi t}} e^{-(v_2-v_1)^2/4t} & \dfrac{\sign (v_2-v_1)}{\sqrt{\pi t}} \cdot \int\limits_{\abs{v_2-v_1}/\sqrt{t}}^{+\infty} e^{-v^2/4}\,dv\\
& & 0 & \dfrac{1}{\sqrt{\pi t}}\\
& & & 0
\end{matrix}
\right]
=
\\
=\dfrac{1}{\pi t}\left(1+\dfrac{\abs{v_2-v_1}}{2\sqrt{t}} \cdot e^{-(v_2-v_1)^2/4t} \cdot \int\limits_{\abs{v_2-v_1}/\sqrt{t}}^{+\infty} e^{-v^2/4}\,dv-e^{-(v_2-v_1)^2/2t}\right).
\end{gather*}
Therefore, computing the integral in~\eqref{equation3} by integrating by parts (several times) and using~\eqref{equation1} and~\eqref{equation2}, we obtain
\begin{equation}
\label{equation4}
\E\nu_t^2([0;u])=1-\dfrac{4}{\pi}+\dfrac{5u}{\sqrt{\pi t}}+\dfrac{u^2}{\pi t}+ \dfrac{4}{\pi}e^{-u^2/2t}-\dfrac{2}{\pi}\int\limits_0^{u/\sqrt{t}} e^{-z^2/4}\,dz-\dfrac{4u}{\pi\sqrt{t}}\int\limits_0^{u/\sqrt{t}}e^{-z^2/2}\,dz.
\end{equation}
Finally, using~\eqref{equation1} and~\eqref{equation4}, we arrive at the desired result.
\end{proof}

\begin{corollary}
\label{corollary5}
The following assertions hold true:
\begin{gather*}
\Var\nu_t([0;u])\sim (3-2\sqrt{2}) \cdot \dfrac{u}{\sqrt{\pi t}},\quad u\to +\infty \text{ or } t\to 0+,\\
\Var\nu_t([0;u])\sim (3-\dfrac{2}{\sqrt{\pi}}) \cdot \dfrac{u}{\sqrt{\pi t}},\quad u\to 0+ \text{ or } t\to +\infty.
\end{gather*}
\end{corollary}

\begin{theorem}
For any $k\ges 1$ we have
\[
\E\nu_t^k([0;u])\sim \left(\dfrac{u}{\sqrt{\pi t}}\right)^k,\quad u\to +\infty \text{ or } t\to 0+.
\]
\end{theorem}

\begin{proof}
Due to the scaling invariance~\eqref{equation0} of the Arratia flow it is enough to prove the corresponding assertion for $t\to 0+$. To do it, we will use induction. For $k=1$ the assertion follows from~\eqref{equation1}. Now suppose that it holds true for all $k'\les k-1$. Then from~\eqref{equation} it follows that
\[
\lim_{t\to 0+} t^{k/2}\E\nu_t^k([0;u])=\lim_{t\to 0+} t^{k/2}\E N_t^{[k]}([0;u]),
\]
provided that the limit on the right-hand side exists. However,
\[
t^{k/2}\E N_t^{[k]}([0;u])=\int\limits_0^u \stackrel{k}{\ldots} \int\limits_0^u \mathrm{Pf}\left[\sqrt{t} \cdot K_t(v_i,v_j),\; i,j=1,\ldots,k\right]\,dv_1 \ldots dv_k,
\]
and the Pfaffian on the right-hand side converges as $t\to 0+$ to the Pfaffian
\[
\mathrm{Pf}
\left[
\begin{matrix}
0 & 1/\sqrt{\pi} & 0 & 0 & 0 & \ldots & 0 & 0 & 0\\
{} & 0 & 0 & 0 & 0 & \ldots & 0 & 0 & 0\\
{} & {} & 0 & 1/\sqrt{\pi} & 0 & \ldots & 0 & 0 & 0\\
{} & {} & {} & 0 & 0 & \ldots & 0 & 0 & 0\\
{} & {} & {} & {} & 0 & \ldots & 0 & 0 & 0\\
{} & {} & {} & {} & {} & \ddots & \vdots & \vdots & \vdots\\
{} & {} & {} & {} & {} & {} & 0 & 0 & 0\\
{} & {} & {} & {} & {} & {} & {} & 0 & 1/\sqrt{\pi}\\
{} & {} & {} & {} & {} & {} & {} & {} & 0\\
\end{matrix}
\right]
=\left(\dfrac{1}{\sqrt{\pi}}\right)^k.
\]
Thus, by the dominated convergence theorem we obtain
\[
\lim_{t\to 0+} t^{k/2}\E N_t^{[k]}([0;u])=\left(\dfrac{u}{\sqrt{\pi}}\right)^k.
\]
The theorem is proved.
\end{proof}

\section{Proof of the main results}
\label{section3}

\begin{proof}[Proof of Theorem~\ref{theorem1}]
Fixing arbitrary $t>0$, let us note that for any $u_1<u_2<u_3$ we have
\begin{equation}
\label{equation5}
\nu_t([u_1;u_3])+1=\nu_t([u_1;u_2])+\nu_t([u_2;u_3]),
\end{equation}
since on the right-hand side the cluster containing the point $x(u_2,t)$ is taken into account twice due to the almost sure continuity of the random mapping $x(\cdot,t) \colon \mbR \rightarrow \mbR$ at the point $u_2$. From~\eqref{equation5} it follows that for all $n\ges 1$
\begin{equation}
\label{equation6}
\nu_t([0;n])-\E\nu_t([0;n])=\sum_{k=1}^n \eta_k,
\end{equation}
where
\[
\eta_k:=\nu_t([k-1;k])-\E\nu_t([k-1;k]),\quad k\ges 1.
\]

Since the stochastic process $\{x(u,t)-u,\; u\in\mbR\}$ is strictly stationary, so is the sequence $\{\eta_n,\; n\ges 1\}$. Now to this sequence we would like to apply the following theorem.

\begin{theorem}
\label{theorem7}
\textup{\cite[Theorem~18.5.3]{IbragimovLinnik}}
Let $\{X_n,\; n\ges 1\}$ be a strictly stationary sequence of centered random variables with finite variance such that
\[
\Var\sum_{k=1}^n X_k\longrightarrow +\infty,\quad n\to +\infty,
\]
and for some $\delta>0$
\[
\E\abs{X_1}^{2+\delta}<+\infty
\]
and
\[
\sum_{n=1}^\infty \left(\alpha^X(n)\right)^{\delta/(2+\delta)}<+\infty,
\]
where $\alpha^X$ is its strong mixing coefficient:
\begin{gather*}
\alpha^X(n):=\sup\{\abs{\mbP(AB)-\mbP(A)\mbP(B)} \mid A\in \sigma(X_j,\; j\les k),\\
B\in\sigma(X_j,\; j\ges k+n),\; k\in\mbZ\},\quad n\in\mbZ,
\end{gather*}
with $\sigma(\mcA)$ standing for the $\sigma$-field generated by the set $\mcA$ of random variables.

Then the series
\[
\E X_1^2+2\sum_{k=2}^\infty \E X_1X_k
\]
is absolutely convergent and, provided that its sum $\sigma^2$ is strictly positive, the following convergence in distribution takes place:
\[
\dfrac{1}{\sqrt{n}} \sum_{k=1}^n X_k \Longrightarrow \mcN(0,\sigma^2),\quad n\to\infty.
\]
\end{theorem}

\begin{remark}
Note that $\sigma^2$ permits the representation
\[
\sigma^2=\lim_{n\to \infty} \dfrac{1}{n} \Var\sum_{k=1}^n X_k,
\]
since
\[
\dfrac{1}{n} \Var\sum_{k=1}^n X_k=\dfrac{1}{n} \E\left(\sum_{k=1}^n X_k\right)^2=\dfrac{1}{n} \sum_{i,j=1}^n \E X_iX_j=\E X_1^2+2\sum_{k=2}^n \dfrac{n-k}{n} \E X_1X_k,
\]
and, if the series $\sum \E X_1X_k$ is absolutely convergent, by the dominated convergence theorem
\[
\lim_{n\to\infty} \sum_{k=2}^n \dfrac{n-k}{n} \E X_1X_k=\sum_{k=2}^\infty \E X_1X_k-\lim_{n\to\infty} \sum_{k=2}^n \dfrac{k}{n} \E X_1X_k=\sum_{k=2}^\infty \E X_1X_k.
\]
\end{remark}

Now let us verify that the conditions of this theorem are satisfied for the sequence $\{\eta_n,\; n\ges 1\}$. First, we note that all absolute moments of $\eta_1$ are finite, since such are those of $\nu_t([0;1])$. Second, from equality~\eqref{equation6} and Corollary~\ref{corollary5} we get
\[
\dfrac{1}{n}\Var\sum_{k=1}^n \eta_k=\dfrac{1}{n}\Var\nu_t([0;n])\longrightarrow \dfrac{3-2\sqrt{2}}{\sqrt{\pi t}}>0,\quad n\to\infty,
\]
and so in particular
\[
\Var\sum_{k=1}^n \eta_k\longrightarrow +\infty,\quad n\to\infty.
\]

Third, it is easy to check that for the strong mixing coefficient $\alpha^\eta$ of the sequence $\{\eta_n,\; n\ges 1\}$ we have
\[
\alpha^\eta(n)\les \alpha(n),\quad n\ges 1,
\]
where
\begin{gather*}
\alpha(n):=\sup\{\abs{\mbP(AB)-\mbP(A)\mbP(B)},\; A\in\sigma(x(u,t)-u,\; u\les h),\\
B\in\sigma(x(u,t)-u,\; u\ges h+n),\; h\in\mbR\}.
\end{gather*}
In~\cite{Glinyanaya} it was proved that for $n\ges 1$ large enough
\[
\alpha(n)\les 2\sqrt{\dfrac{2}{\pi t}} \int\limits_n^{+\infty} e^{-r^2/2t}\,dr.
\]
Therefore, using the standard estimate for the tails of the Gaussian distribution, we obtain that for $n\ges 1$ large enough
\[
\alpha^\eta(n)\les 2\sqrt{\dfrac{2}{\pi t}} \int\limits_n^{+\infty} e^{-r^2/2t}\,dr\les \dfrac{2}{n}\sqrt{\dfrac{2}{\pi t}} e^{-n^2/2t},
\]
and so for all $\delta>0$
\[
\sum_{n=1}^\infty \left(\alpha^\eta(n)\right)^{\delta/(2+\delta)}<+\infty.
\]

Thus, applying Theorem~\ref{theorem7} to the sequence $\{\eta_n,\; n\ges 1\}$ and using equality~\eqref{equation6} finishes the proof.
\end{proof}

\begin{proof}[Proof of Theorem~\ref{theorem2}]
The proof is based on the following theorem.

\begin{theorem}
\textup{\cite[Theorem~2]{Tikhomirov}}
Let $\{X_n,\; n\ges 1\}$ be a strictly stationary sequence of centered random variables with finite variance such that for some $\delta\in (0,1]$
\[
\E\abs{X_1}^{2+\delta}<+\infty
\]
and for some constants $K>0$ and $\beta>0$
\[
\alpha^X(n)\les Ke^{-\beta n},\quad n\ges 1.
\]

Then there exists a constant $A=A(K,\beta,\delta)>0$ such that
\[
\sup_{z\in\mbR} \abs{\Prob{\frac 1{\sigma_n} \sum_{k=1}^n X_k\les z}- \dfrac{1}{\sqrt{2\pi}} \int\limits_{-\infty}^z e^{-r^2/2}\,dr}\les An^{-\delta/2}(\log n)^{1+\delta},\quad n\ges 1,
\]
where
\[
\sigma_n^2=\E\left(\sum_{k=1}^n X_k\right)^2.
\]
\end{theorem}

Applying this theorem to the sequence $\{\eta_n,\; n\ges 1\}$ defined above and using equality~\eqref{equation6}, we obtain the desired result.
\end{proof}

\end{document}